\documentclass{amsart}
\usepackage{enumerate}
\usepackage{colortbl}
\newtheorem{theorem}{Theorem}[section]
\newtheorem{lemma}[theorem]{Lemma}
\theoremstyle{definition}
\newtheorem{definition}[theorem]{Definition}
\newtheorem{example}[theorem]{Example}

\newtheorem{corollary}[theorem]{Corollary}

\theoremstyle{remark}

\numberwithin{equation}{section}

\begin{document}

\def\C{\mathbb C}
\def\R{\mathbb R}
\def\X{\mathbb X}
\def\Z{\mathbb Z}
\def\Y{\mathbb Y}
\def\Z{\mathbb Z}
\def\N{\mathbb N}
\def\cal{\mathcal}

\def\cal{\mathcal}
\def\b{\mathcal B}
\def\c{\mathcal C}
\def\cc{\mathbb C}
\def\x{\mathbb X}
\def\r{\mathbb R}
\def\uu{(U(t,s))_{t\ge s}}
\def\vv{(V(t,s))_{t\ge s}}
\def\xx{(X(t,s))_{t\ge s}}
\def\yy{(Y(t,s))_{t\ge s}}
\def\zz{(Z(t,s))_{t\ge s}}
\def\ss{(S(t))_{t\ge 0}}
\def\tt{(T(t,s))_{t\ge s}}
\def\rr{(R(t))_{t\ge 0}}

\title[Fractional Differential Equations]{
A Simple Spectral Theory of Polynomially Bounded Solutions and Applications to Differential Equations}

\author{Vu Trong Luong}
\address{VNU University of Education, Vietnam National University, Hanoi; 144 Xuan Thuy, Cau Giay, Hanoi, Vietnam}
\email{vutrongluong@gmail.com}

\author{Nguyen Van Minh}

\address{Department of Mathematics and Statistics, University of Arkansas at Little Rock, 2801 S University Ave, Little Rock, AR 72204. USA}
\email{mvnguyen1@ualr.edu}

\thanks{This paper is supported by the grant B2019-TTB-01 of the Ministry of Education and Training of Vietnam}
\thanks{The authors are grateful to the anonymous referee for his carefully reading the manuscript and useful suggestions.}
\date{\today}

\begin{abstract} In this paper we present a simple spectral theory of polynomially bounded functions on the half line, and then apply it to study the asymptotic behavior of solutions of fractional differential equations of the form $D^{\alpha}_Cu(t)=Au(t)+f(t)$, where $D^{\alpha}_Cu(t)$ is the derivative of the function $u$ in Caputo's sense, $A$ is generally an unbounded closed operator, $f$ is polynomially bounded. Our main result claims that if $u$ is a mild solution of the Cauchy problem such that $\lim_{h\downarrow 0} \sup_{t\ge 0} \| u(t+h)-u(t)\|/(1+t)^n=0$, and $\sup_{t\ge 0} \| u(t)\| /(1+t)^n <\infty$, then, $\lim_{t\to\infty} u(t)/(1+t)^n  =0$ provided that the spectral set
$\Sigma (A,\alpha )\cap i\R$ is countable, where $\Sigma (A,\alpha )$ is defined to be the set of complex numbers $\xi$ such that $\lambda^{\alpha -1} (\lambda^\alpha -A)^{-1}$ is analytic in a neighborhood of $\xi$, and $u$ satisfies some ergodic conditions with zero means. The obtained result extends known results on strong stability of solutions to fractional equations.
\end{abstract}
\maketitle

\begin{center}
{\it Dedicated to the 75th anniversary of the birthday of Professor Toshiki Naito}
\end{center}

\section{Introduction} \label{section 1}
In this paper we deal with the asymptotic behavior of solutions to linear fractional differential equations the form
\begin{equation}\label{eq:0}
D^{\alpha}_Cu(t)=Au(t)+f(t), u(0)=x, 0<\alpha\le1,
\end{equation}
where $D^{\alpha}_Cu(t)$ is the derivative of the function $u$ in the Caputo's sense. 

\medskip
In recent decades fractional differential equations are of increasing interests to many researchers as this kind of equations allows us to model complex processes. In the model using these equations one can take into account of nonlocal relations in space as well as in time. Due to these properties fractional differential equations have been applied extensively in engineering. We refer the reader to the monographs \cite{hil,kilsritru} for an account of applications in Physics and Engineering. For general results and concepts in abstract spaces the reader is referred to \cite{baemeenan,baz,clegrilon}. In recent years the asymptotic behavior of mild solutions of fractional differential equations are extensively studied. Among many results we would like to mention \cite{cue,keylizwar,lizngu,lizngu2} that deal with existence, uniqueness of mild solutions as well as their asymptotic behavior. In this short paper we would like to extend a famous result on stability of $C_0$-semigroups due to Sklyar-Shirman \cite{sklshi} in the bounded case, and Arendt-Batty \cite{arebat}, Lyubich-Vu \cite{lyuvu} in the general unbounded case. Many extensions of this result (that is now referred to as ABLV Theorem) are given in \cite{arebathieneu,batengprusch,chitom1,nee,batengprusch,pau}, see also \cite{hinnaiminshi,min,minngusie} for related results on applications of spectral theory of functions to the study of the asymptotic behavior of solutions. One of interesting directions of extension of this theorem is  individual versions for stability of bounded mild solutions of evolution equations of the form $u'(t)=Au(t)+f(t), t\ge 0$ (see e.g. \cite{batneerab1,batneerab2}, \cite[Theorem 5.3.6]{nee}). Further, stability with weight of individual orbits of $C_0$-semigroups (or more general objects as representations) is also studied in \cite{batyea,baspry} with general conditions on weights.
In this direction, a concept of spectrum of a bounded function on the half line is introduced based on its Laplace transform (that is defined to be the subset of the imaginary axe in the complex plane where the Laplace transform of the function has no analytic extension through). 

\medskip
In this paper we choose a simple approach that is based on the analysis of the set of solutions of ordinary differential equations $y'=\lambda y +g(t)$ in order to examine the resolvent $R(\lambda , \tilde{\cal D})$ of the operator induced by the differentiation operator $d/dt$ in a quotient space. This allows us to define the spectrum of a polynomially bounded function $g$ on the half line. Our main result is Theorem \ref{the main} that is an extension of the ABLV Theorem for polynomially bounded mild solutions and for fractional evolution equations in Banach spaces. When Eq.(\ref{eq:0}) is homogeneous (that is, $f=0$), and the homogeneous equation is well posed with $\alpha =1$ our results have some overlaps with some results obtained in \cite{batyea,baspry}. Otherwise, to our best knowledge, the obtained results of this paper is new, in addition to its simple approach via the differentiation operator $d/dt$.

\section{Preliminaries and Notations}
Throughout this paper we will denote by $\R,\R^+, \C$ the real line $(-\infty,\infty)$, half line $[0,+\infty )$ and the complex plane. For $J$ being either $\R$ or $\R^+$, the notation $BUC(J,\X)$ stands for the function space of all bounded and uniformly continuous functions taking values in a (complex) Banach space $\X$ with sup-norm. Below we denote
 $$g_\alpha(t)=\frac{t^{\alpha-1}}{\Gamma(\alpha)}, t>0,\alpha>0.$$
For a complex number $z$, $\Re z$ denotes its real part. In this paper the single valued power function $\lambda^\alpha$ of the complex variable $\lambda$ is uniquely defined as
$\lambda^\alpha =| \lambda | ^\alpha e^{ i \alpha\ arg (\lambda)}$, with $-\pi < arg (\lambda) < \pi$.

\subsection{Fractional differentiation in Caputo's sense}

Let $\alpha >0, t\ge a,$ and $a$ is a fixed number. Then, the fractional operator
\begin{align}
J^{\alpha}_a u(t):=(g_\alpha \ast u)(t)=\displaystyle\int_{a}^{t}g_\alpha(t-\tau)u(\tau)d\tau
\end{align}
is called fractional Riemann-Liouville integral of degree $\alpha$. The function
\begin{equation*}
D^{\alpha}_{C}u(t):=
\begin{cases}
J^{n-\alpha}u^{(n)}(t)=\dfrac{1}{\Gamma(n-\alpha)}\displaystyle\int_{a}^t\dfrac{u^{(n)}(\tau)}{(t-\tau)^{\alpha+1-n}}d\tau, &{n-1}<{\alpha}<n\in\mathbb{N},\\
u^{(n)}(t),&\alpha=n\in\mathbb{N},
\end{cases}
\end{equation*}
 is called the fractional derivative in Caputo's sense of degree $\alpha$. By this notation we have for $0<\alpha \le 1$
 $$J^\alpha _a D^{\alpha}_C u(t)=u(t)-u(a).$$

\subsection{Cauchy Problem}
For a fixed $0<\alpha \le 1,$ consider the Cauchy problem
\begin{equation}\label{eq:1}
D^{\alpha}_Cu(t)=Au(t), u(0)=x, 
\end{equation}
where $A$ is generally an unbounded linear operator.

\medskip
The well-posedness of \eqref{eq:1}  is equivalent to that of the problem
\begin{equation}\label{eq:4}
u(t)=x+\int_0^tg_\alpha(t-s)Au(s)ds.
\end{equation}
The reader is referred to the monograph \cite{pru} for an extensive study of the well-posedness of this kind of equations when $A$ is generally an unbounded operator. Recent extensions for more general equations can be found in \cite{keylizwar} and their references for more details.

\medskip
Let us consider inhomogeneous linear equations of the form
\begin{equation}\label{eq:2a}
D^{\alpha}_Cu(t)=Au(t)+f(t), t\geq 0.
\end{equation}

\begin{definition}
A mild solution $u$ of Eq.(\ref{eq:2a}) on $\R^+$ is a continuous function on $\R^+$ such that, for each $t\in\R^+$, $J^\alpha u(t) \in D(A)$ and
$$u(t)=AJ^\alpha u(t)+J^\alpha f(t)+u(0).$$
\end{definition}

For a fixed integer $n  \ge 0$ we will use $BC_n  (\R^+,\X)$ to denote the space of all continuous function on $\R^+$ with values in $\X$ such that 
\begin{equation}\label{2.2}
\sup_{t\in \R^+} \frac{\| f(t)\|}{(1+t)^n } <\infty ,
\end{equation}
and the norm of an element $f\in BC_n  (\R^+,\X)$ is defined to be (\ref{2.2}). Every function satisfying (\ref{2.2}) is called $n $-bounded.
It is easy to see that $BC_n (\R^+,\X)$ with norm
\begin{equation}
\| f\|_n  := \sup_{t\in \R^+} \frac{\| f(t)\|}{(1+t)^n } 
\end{equation}
becomes a normed space.

\begin{definition}
We say that a function $f:\R^+ \to \X$ is $n $-uniformly continuous if it is continuous and
\begin{equation}
\lim_{h\downarrow 0} \sup_{t\in\R^+} \frac{\| f(t+h)-f(t)\|} {(1+t)^n  } =0.
\end{equation}
\end{definition}
We denote by $BUC_n  (\R^+,\X)$ as the part of $BC_n  (\R^+,\X)$ consisting of all $n $-uniformly continuous functions from $\R^+$ to $\X$.

\begin{lemma}
The normed space $(BUC_n  (\R^+,\X),\| \cdot \|_n  )$ is complete, so it is a Banach space.
\end{lemma}
\begin{proof}
We will make use of the fact that is widely known in the literature that the function space $BC(\R^+,\X)$ with sup-norm is a Banach space. Therefore, if $\{f_k\}_{k=1}^\infty$ is a Cauchy sequence in $BUC_n  (\R,\X)$, then the sequence $\{g_k\}_{k=1}^\infty$, where $g_k(t):= f_k(t)/(1+t)^n $, is a Cauchy sequence in $BC(\R^+,\X)$, so it is convergent to a function $g\in BC(\R^+,\X)$. We are going to show that if $f$ is defined as $f(t)=(1+t)^n  g(t)$, then $f\in BUC_n  (\R^+,\X)$ and $f$ is the limit of $\{f_k\}_{k=1}^\infty $. Indeed, by assumption,
\begin{eqnarray}\label{2.4}
\lim_{k\to\infty} \| f_k -f\|_n &=& \lim_{n\to\infty} \sup_{t\in\R^+} \| \frac{f_n(t)-f(t)}{(1+t)^n  }\| \nonumber \\
&=& \lim_{k\to\infty} \sup_{t\in\R^+} \| g_k(t)-g(t)\|  \nonumber  \\
&=& \lim_{k\to\infty}\| g_k-g\|  \nonumber  \\
&=& 0.
\end{eqnarray}
This yields that
$$
\sup_{t\in\R^+}\frac{ \| f(t)\|}{(1+t)^n } <\infty .
$$
Next, we will show that $f$ is $n $-uniformly continuous. In fact, for each $k\in\N$, we have
\begin{eqnarray*}
  \frac{\| f(t+h)-f(t)\|}{(1+t)^n } \le \frac{\| f(t+h)-f_k(t+h)\|}{(1+t)^n  } + \frac{\| f_k(t+h)-f_k(t)\|}{(1+t)^n }  + \frac{\| f_k(t)-f(t)\|}{(1+t)^n }.
\end{eqnarray*}
For every given $\epsilon >0$ we can find a (fixed) sufficiently large $N$ such that 
$$
 \sup_{t\in\R^+}\frac{\| f_N(t)-f(t)\|}{(1+t)^n } < \frac{\epsilon}{6} .
 $$
By the $n $-uniformness of $f_N$ there exists a positive $\delta$ such that if $0<h<\delta$, then
$$
\sup_{t\in\R^+}\frac{\| f_N(t+h)-f_N(t)\|}{(1+t)^n } <\frac{\epsilon}{6}.
$$
Next, for $0<h<\delta$
\begin{eqnarray*}
\frac{\| f(t+h)-f_N(t+h)\|}{(1+t)^n  }  &=& \frac{\| f(t+h)-f_N(t+h)\|} {(1+t+h)^n }\cdot \frac{(1+t+h)^n } {(1+t)^n  } \\
&\le& \sup_{t\in\R^+} \frac{\| f(t)-f_N(t)\|} {(1+t)^n }\cdot \frac{(1+t+\delta )^n } {(1+t)^n  } \\
&=& \sup_{t\in\R^+} \frac{\| f(t)-f_N(t)\|} {(1+t)^n }\cdot (1+\delta)^n  \\
&\le& \frac{\epsilon}{6} (1+\delta)^n  .
\end{eqnarray*}
If we choose $\delta$ sufficiently small, say $\delta <\delta_0:= 2^{1/n }-1$, then, $(1+\delta )^n  < 2$, so
 for $0<h<\delta$,
\begin{equation}
\sup_{t\in\R^+} \frac{\| f(t+h)-f_N(t+h)\|}{(1+t)^n  }  <\epsilon .
\end{equation}
This yields $f\in BUC_n (\R^+,\X)$ and by (\ref{2.4})  it is the limit of $\{f_n\}_{n=1}^\infty$. The lemma is proved.
\end{proof}

\begin{example}\label{exa 1}
Let $f\in BC_n  (\R^+,\X)$. If its derivative $f'$ is also an element of $BC_n  (\R^+,\X)$, then, $f\in BUC_n  (\R^+,\X)$.
\end{example}
\begin{proof}
For all $t\in \R^+$ and $h>0$, we have
\begin{eqnarray*}
\lim_{h\downarrow 0} \sup_{t\in \R^+} \frac{\| f(t+h)-f(t)\| }{(1+t)^n }&\le& 
\lim_{h\downarrow 0} \sup_{t\in \R^+}\frac{\sup_{t\le \xi \le t+h} \| f'(\xi )\| \cdot h}{(1+t)^n }  \\
\end{eqnarray*}
Note that by the $n $-boundedness of $f'$,
\begin{eqnarray*}
\frac{\sup_{t\le \xi \le t+h} \| f'(\xi )\| }{(1+t)^n }  &=& \frac{(1+t+h)^n }{(1+t)^n }\frac{\sup_{t\le \xi \le t+h} \| f'(\xi )\| }{(1+t+h)^n } \\
&\le& \frac{(1+t+h)^n }{(1+t)^n } \sup_{t\le \xi \le t+h}\frac{ \| f'(\xi )\| }{(1+\xi )^n } \\
&=&( 1+h)^n  \| f'\|_n  .
\end{eqnarray*}
Therefore,
\begin{eqnarray*}
\lim_{h\downarrow 0} \sup_{t\in \R^+}\frac{\sup_{t\le \xi \le t+h} \| f'(\xi )\| }{(1+t)^n } \cdot h 
= \lim_{h\downarrow 0} h( 1+h)^n  \| f'\|_n  =0.
\end{eqnarray*}
This completes the proof of the example's claim.
\end{proof}

\section{A Spectral Theory of Polynomially bounded Functions}
We note that for every function $f\in BUC_n  (\R^+,\X)$ its Laplace transform 
$$
{\cal L} f (\lambda ) := \int^\infty _0 e^{-\lambda t}f(t)dt 
$$
 exists for any $\Re \lambda > 0$, so the definition of the spectrum $Sp_+(f)$  as the set of all reals $\xi_0$ such that its Laplace transform has no analytic extension to any neighborhood of $i\xi_0$ as in \cite{arebat} can be {\bf formally} extended to $f\in BUC_n(\R^+,\X)$. The problem is how this spectrum can control the asymptotic behavior of the function $f$ on the half line $\R^+$ is not clear due to the unboundedness of polynomially bounded functions $f$. In what follows we will discuss an approach to the concept of spectrum of $f$ and how under some further "ergodic" conditions it controls the behavior of the functions $f\in BUC_n(\R^+,\X)$. We will begin with the translation semigroup $(S(t)_{t\ge 0})$ in $BUC_n(\R^+,\X)$, i.e., $S(t)f:= f(t+\cdot )$ for each $f\in BUC_n(\R^+,\X)$.

\begin{lemma}\label{lem 3.1}
For each $t\ge 0$, we have
\begin{eqnarray}
\| S(t) \| _n 
&\le& (1+t)^n   .
\end{eqnarray}
\end{lemma}
 \begin{proof}
 For each $f\in BUC_n  (\R^+,\X)$ we have
 \begin{eqnarray*}
\| S(t)f \| _n  &=& \sup_{s\ge 0}   \frac{\| f(t+s)\|}{(1+s)^n }       \\
&=& \sup_{s\ge 0} \left(    \frac{\| f(t+s)\|}{[1+(t+s)]^n }   \cdot  \frac{[1+(t+s)]^n  }{(1+s)^n }  \right)  \\
&\le & \| f\| _n   \cdot     \sup_{s\ge 0}   \frac{[1+(t+s)]^n  }{(1+s)^n }\\
&=& \| f\| _n   \cdot   \sup_{s\ge 0} \left( 1+\frac{t}{1+s} \right)^n  \\
&\le & (1+t)^n  \| f\|_n  .
\end{eqnarray*}
 \end{proof}

Let $\cal D$ denote the differentiation operator $d/dt$ in $BUC_n  (\R^+,\X)$ with domain
$$
D(\cal D) =\{ f\in BUC_n  (\R^+,\X): \exists f', \ f'\in BUC_n  (\R,\X)\} .
$$

 \begin{lemma}
 The following assertions are valid:
 \begin{enumerate}
  \item The translation semigroup $(S(t)_{t\ge 0})$ in $BUC_n  (\R^+,\X)$ is strongly continuous;
 \item
 The infinitesimal generator $\cal G$ of $(S(t)_{t\ge 0})$ is the differentiation operator ${\cal D}$ in $BUC_n  (\R^+,\X)$.  \end{enumerate}
 \end{lemma}
 \begin{proof}
 As the proof can be done in a standard manner it is omitted.
 \end{proof}

\subsection{Operator $\tilde{\cal D}$}
Throughout the paper we will use the following notation
\begin{eqnarray*}
C_{0,n }(\R^+,\X) &:=&\{ f\in BUC_n  (\R^+,\X): \ \lim_{t\to \infty}\| f(t)\| /(1+t)^n =0\} .
\end{eqnarray*}
It is easy to see that $C_{0,n }(\R^+,\X)$ is a closed subspace $BUC_n  (\R^+,\X)$, and is invariant under the translation semigroup $(S(t)_{t\ge 0})$.
In the space $BUC_n  (\R^+,\X)$ we introduce the following relation $R$:
\begin{equation}
 f \ R \ g \ \mbox{if and only if} \ \ f -g \in C_{0,n }(\R^+,\X) .
\end{equation}
This is an equivalence relation and the quotient space $\Y:= BUC_n  (\R^+,\X)/ R$ is a Banach space. We will also denote the norm in this quotient space $\Y$ by $\| \cdot \|_n$ if it does not cause any confusion.

\medskip
The class containing $f\in BUC_n  (\R^+,\X)$ will be denoted by $\tilde{f}$. Define $\tilde{\cal D}$ in $ \Y=BUC_n  (\R^+,\X)/ R$ as follows:
\begin{eqnarray}
D( \tilde{\cal D}) &:=&\{ \tilde{f} \in \Y : \exists u\in \tilde{f}, u\in D(\cal D)\} .
\end{eqnarray}
If $f\in D( \tilde{\cal D}) $, we set
\begin{equation}
\tilde{D}\tilde{f} := \widetilde {\cal Du} 
\end{equation}
for some $u\in \tilde f$. The following lemma will show that this $\tilde{D}$ is well defined as an operator in $\Y$.
\begin{lemma}
With the above notations, $\tilde{\cal D}$ is a well defined single valued linear operator in $\Y$. \end{lemma}
\begin{proof}
First we show that the operator is a well defined single-valued operator. In fact, assuming $\tilde{f}\in D(\tilde{\cal D})$, we will prove that the definition of $\tilde{\cal D}\tilde{f}$ does not depend on the choice of representatives $u$ of this class $\tilde{f}$. To this end, suppose that $u,v\in \tilde{f}$ such that $u,v\in D(\cal D)$. Then, by the definition of $\tilde{\cal D}\tilde{f}$, $\tilde{\cal D}\tilde{f} =\widetilde{{\cal D}u}$, and at the same time $\tilde{\cal D}\tilde{f} =\widetilde{{\cal D}v}$. 
We will show that $\widetilde{{\cal D}u}= \widetilde{{\cal D}v}$, or equivalently, ${\cal D}(u-v) \in C_{0,n } (\R^+,\X)$. In fact, since $u,v\in \tilde{f}$, if we set $h:= u-v$, then $h \in C_{0,n }(\R^+,\X)$, and $h\in D({\cal D})$.
 Therefore,
$$
\lim_{t\to 0^+} \frac{S(t)h-h}{t} = {\cal D}h.
$$
Note that both $S(t)h$ and $h$ are in $C_{0,n }(\R^+,\X)$, so is ${\cal D}h={\cal D}(u-v)$. This proves that $\tilde{\cal D}$ is a well defined single valued operator. Its linearity is clear. The lemma is proved.
\end{proof}

For each given $f\in BUC_n  (\R^+,\X)$ consider the following complex function $\hat f(\lambda )$ in $\lambda$ defined as
\begin{equation}
\hat f (\lambda ) := (\lambda -\tilde{\cal D})^{-1} \tilde{f} .
\end{equation}

\begin{lemma}\label{lem 3.5}
$\hat f(\lambda )$ exists as an analytic function of $\lambda$ in the region $\lambda \in \C \backslash i\R$. Moreover, the following inequality is valid
\begin{eqnarray}
\| (\lambda -\tilde{\cal D})^{-1} \tilde{f}\|_n  &\le  &
\frac{e^{\Re \lambda} \Gamma (n+1,\Re \lambda ) }{(\Re \lambda)^{n+1} }
\| \tilde f\|_n  , \ \mbox{for all} \  \Re\lambda >0 ,\\
\| (\lambda -\tilde{\cal D})^{-1} \tilde{f}\|_n  &\le  & \frac{\|f\|_n}{|\Re\lambda |} , \ \mbox{for all} \ \Re\lambda <0 ,
\end{eqnarray}
where
$$
\Gamma (a,x):= \int^\infty _x t^{a-1}e^{-t}dt , \ a, x \ge 0 . 
$$
\end{lemma}
\begin{proof}
Consider the equation
\begin{equation}
y'=\lambda y + f(t), \ t\ge 0.
\end{equation}
{\bf The case when $\Re\lambda <0$:} the general solution of the equation in this case is
\begin{equation}
y(t)= e^{\lambda t} y_0 +\int^t_{0} e^{\lambda (t-s)}f(s)ds , \ y_0\in \X, t\ge 0.
\end{equation}
 Since $\Re\lambda <0$ all functions $g(t)=e^{\lambda t}y_0$ are exponentially decay to zero, so they are all in $C_{0,n  }(\R^+,\X)$. We are going to show that in this case ($\Re\lambda <0$), $\lambda \in \rho (\tilde{\cal D})$, and the function
 $$
 u: \R^+ \ni t \mapsto \int^t_{0} e^{\lambda (t-s)}f(s)ds 
 $$
 is a representative of
 the equivalence class $(\lambda -\tilde{\cal D})^{-1} \tilde f$. To this purpose, let $g\in C_{0,n } (\R^+,\X)$. Then, the general solution of the equation
 $$
 y'=\lambda y+(f(t)+g(t))
 $$
 is the following
 \begin{eqnarray*}
y(t) &=& e^{\lambda t} y_0 +\int^t_{0} e^{\lambda (t-s)}(f(s)+g(s))ds \\
&=& e^{\lambda t} y_0 + \int^t_{0} e^{\lambda (t-s)}g(s)ds +\int^t_{0} e^{\lambda (t-s)}f(s)ds.
\end{eqnarray*}
 
It suffices to show that the function $h$ defined as
 $$
 h(t):=  \int^t_{0} e^{\lambda (t-s)}g(s)ds  ,
 $$
 is in $C_{0,n  }(\R^+,\X)$. In fact, for a given $\epsilon' >0$ there is a sufficiently large number $T_0$ such that for $t \ge T_0$ 
 $$ \frac{ \| g(t)\|}{(1+t)^n  } < \epsilon' .$$
Next, for $t\ge T_0$ we have
\begin{eqnarray*}
\frac{\| \int^t_{0} e^{\lambda (t-s)}g(s)ds \|}{(1+t)^n } &\le & \frac{1}{(1+t)^n } \left( \int^{T_0}_0 e^{\Re \lambda (t-s)} \| g(s\| ds + \int^t_{T_0} e^{\Re\lambda (t-s)} \| g(s)\| ds \right) \\
&\le & \frac{e^{\Re\lambda t}}{(1+t)^n } \int^{T_0}_0 e^{-\Re\lambda s}\|g(s)\| ds + \int^t_{T_0} e^{\Re\lambda (t-s)}\frac{\|g(s)\|} {(1+s)^n  }ds\\
\end{eqnarray*}
Since $\Re\lambda <0$,
$$
\lim_{t\to \infty}  \frac{e^{\Re\lambda t}}{(1+t)^n } \int^{T_0}_0 e^{-\Re\lambda s}\|g(s)\| ds =0.
$$
For any given $\epsilon >0$, there is a sufficiently large number $T_1$ such that if $t\ge T_1$, then
$$
 \frac{e^{\Re\lambda t}}{(1+t)^n } \int^{T_0}_0 e^{-\Re\lambda s}\|g(s)\| ds < \frac{\epsilon}{2}.
$$
On the other hand,
\begin{eqnarray*}
 \int^t_{T_0} e^{\Re\lambda (t-s)}\frac{\|g(s)\|} {(1+s)^n  }ds &\le& e^{\Re\lambda t} \int^t_{T_0} e^{-\Re\lambda s}\epsilon' ds \\
 &=& \frac{\epsilon'}{-\Re \lambda } (1-e^{\Re\lambda(t-T_0)}) \\
 &\le&  \frac{\epsilon'}{-\Re \lambda } .
\end{eqnarray*}
Finally, for any given $\epsilon >0$, if we choose $T=\max (T_0,T_1)$ and $\epsilon' = -\epsilon \Re\lambda /2$, then, for $t\ge T$
\begin{equation}
\frac{\|h(t)\|}{(1+t)^n  } < \epsilon .
\end{equation}
This shows that $h\in C_{0,n }(\R^+,\X)$.

\medskip
We will show that the function $u$ defined as
$$
 u(t):=  \int^t_{0} e^{\lambda (t-s)}f(s)ds .
 $$
 is in $BUC_n  (\R^+,\X)$ for every $\Re\lambda <0$, and $f\in BUC_n  (\R^+,\X)$. In fact, 
 as noted in the Example (\ref{exa 1})
 $u$ is a solution of the equation
 $u'=\lambda u+f$, so this claim is proved if we can show that $u$ is $n $-bounded. Since $\Re\lambda <0$, we have
 \begin{eqnarray}
\sup_{t\in\R^+} \frac{ \| u(t) \|}{(1+t)^n } &\le & \sup_{t\in\R^+} \frac{ \int^t_0 e^{\Re\lambda (t-s)}\| f(s)\| ds}{(1+t)^n }  \nonumber\\
&\le& \sup_{t\in\R^+} \int^t_0 e^{\Re\lambda (t-s)} \frac{\| f(s)\|}{(1+s)^n } ds \nonumber \\
&=&\sup_{t\in\R^+}\int^t_0 e^{\Re\lambda (t-s)} \| f(s)\|_n  ds \nonumber \\
&=&\sup_{t\in\R^+} \frac{\|f\|_n }{|\Re\lambda|} (1-e^{\Re\lambda t}) \nonumber\\
&\le& \frac{\|f\|_n }{|\Re\lambda|} .\label{norm of D}
 \end{eqnarray}
 By the above argument, we have proved that if $\Re\lambda <0$, $\lambda \in \rho (\tilde{\cal D})$, and
 $$
 (\lambda -\tilde{\cal D})^{-1} \tilde{f} =\tilde u
 $$
and, by (\ref{norm of D})
 \begin{eqnarray*}
\| (\lambda -\tilde{D})^{-1}\tilde{f} \|_n  &\le&     \frac{\|f\|_n  }{|\Re\lambda|} .
 \end{eqnarray*}

 \bigskip\noindent
{\bf The case when  $\Re \lambda >0$:} the general solution of the equation 
\begin{equation}\label{3.24}
y'=\lambda y+f(t),
\end{equation}
where $y(t)\in \X$, and $f\in BUC_n  (\R^+,\X)$,
is
\begin{equation}
y(t)= e^{\lambda t}y_0 -\int^\infty_t e^{\lambda (t-s)}f(s)ds , \  y_0 \in \X , t\ge 0.
\end{equation}
Note that the function $h(t):= e^{\lambda t}\| y_0\| $ for non-zero $y_0$ grows exponentially as $t\to \infty$. Let us consider the function
\begin{equation}\label{def of g}
g(t):=  \int^\infty_t e^{\lambda (t-s)}f(s)ds .
\end{equation}
We will show that $g$ is the only $n $ bounded solution. By using the Laplace transform of the function $f(t)=(1+t)^n$ we have
\begin{eqnarray}
\frac{\| g(t)\|}{(1+t)^n  } &\le& \frac{1}{(1+t)^n }  \int^\infty_t e^{\Re\lambda (t-s)}\frac{\| f(s)\|}{(1+s)^n } (1+s)^n  ds \nonumber \\
&\le&  \frac{\| f\|_n }{(1+t)^n }  \int^\infty_t e^{\Re \lambda (t-s)}  (1+s)^n  ds \nonumber  \\
&=& \frac{\| f\|_n }{(1+t)^n } \int^\infty_0 e^{-\Re\lambda \xi }(1+t+\xi )^n  d\xi \nonumber  \\
&=& \| f\|_n  \int^\infty_0 e^{-\Re\lambda \xi }\frac{(1+t+\xi )^n }{(1+t)^n } d\xi \nonumber  \\
&\le& \| f\|_n  \int^\infty_0 e^{-\Re\lambda \xi } (1+\xi )^nd\xi \nonumber \\
&=& \frac{e^{\Re \lambda} \Gamma (n+1,\Re \lambda ) }{(\Re \lambda)^{n+1} }
 \| f\|_n  .\label{c3.14}
\end{eqnarray}

This means, $g$ is the unique solution in $BUC_n  (\R^+,\X)$ of (\ref{3.24}), and thus $(\lambda -{\cal D})^{-1}f  =g$. 
Next, we show that if $h\in C_{0,n }(\R^+,\X)$, then the function $k(\cdot)$ defined as
$$
k(t):= \int^\infty_t e^{\lambda (t-s)}h(s)ds 
$$
is in $ C_{0,n }(\R^+,\X)$. In fact, for each given $\epsilon >0$, there exists a number $N>0$ such that if $t>N$ then, 
$$
\frac{\| h(t)\| }{(1+t)^n} <\epsilon .
$$
Consequently, for $t>N$, in the same way as above we have
\begin{eqnarray}
\frac{\| k(t)\|}{(1+t)^n  } &\le& \frac{1}{(1+t)^n }  \int^\infty_t e^{\Re\lambda (t-s)}\frac{\| h(s)\|}{(1+s)^n } (1+s)^n  ds \nonumber \\
&\le&  \frac{\epsilon }{(1+t)^n }  \int^\infty_t e^{\Re \lambda (t-s)}  (1+s)^n  ds \nonumber  \\
&=& \frac{\epsilon }{(1+t)^n } \int^\infty_0 e^{-\Re\lambda \xi }(1+t+\xi )^n  d\xi \nonumber  \\
&=&  \epsilon  \int^\infty_0 e^{-\Re\lambda \xi }\frac{(1+t+\xi )^n }{(1+t)^n } d\xi \nonumber  \\
&\le&  \epsilon   \int^\infty_0 e^{-\Re\lambda \xi } (1+\xi )^nd\xi \nonumber \\
&\le&  \epsilon \frac{e^{\Re \lambda} \Gamma (n+1,\Re \lambda ) }{(\Re \lambda)^{n+1} }.
\end{eqnarray}
This shows that $k\in C_{0,n }(\R^+,\X)$.

\medskip
At this point we have shown that the map $(\lambda -{\cal D})^{-1} $ taking $f\in BUC_n(\R^+,\X)$ into the function $g$ is well defined as the inverse of $(\lambda -{\cal D})$ and it leaves $C_{0,n }(\R^+,\X)$ invariant.
Consequently, if $\Re\lambda >0$, $\lambda \in \rho (\tilde D)$, and  
\begin{eqnarray}
\| (\lambda -\tilde{\cal D})^{-1}\tilde{f} \|_n   &\le& \| g\|_n \nonumber  \\
&\le& \frac{e^{\Re \lambda} \Gamma (n+1,\Re \lambda ) }{(\Re \lambda)^{n+1} }
  \| f\|_n .
 \end{eqnarray}
 As this is valid for any representative in the class $\tilde f$, we have
 \begin{eqnarray}
\| (\lambda -\tilde{\cal D})^{-1}\tilde{f} \|_n   &\le&  \frac{e^{\Re \lambda} \Gamma (n+1,\Re \lambda ) }{(\Re \lambda)^{n+1} }
  \| \tilde f\|_n .
 \end{eqnarray}
 This completes the proof of the lemma.
 \end{proof}

\begin{definition}
Let $n$ be a fixed natural number, and let $f\in BUC_n  (\R^+,\X)$. The set of all points $\xi_0 \in \R$ such that $\hat f (\lambda )$ has no analytic extension to any neighborhood of $i\xi_0$ is defined to be the spectrum of $f$, denoted by $\sigma_n  (f)$
\end{definition}
We will need the following result that is closely related to \cite[Chap. 0]{pru},
\cite[Lemma 2.17]{min}:

\begin{lemma}\label{lem prep gel}
Let $N$ be a natural number, and let $f(z)$ be a complex function taking values in a Banach space $\X$ and be holomorphic in $\C \backslash i\R$ such that there is a positive number $M$ independent of $z$ for which
\begin{equation}
\| f(z)\| \le \frac{M}{|\Re\ z|^N}, \quad \ \mbox{for
all} \ \Re z \not= 0 , |\Re z|< 1.
\end{equation}
Assume further that $i\xi\in i\R$ is an isolated singular point of $f(z)$ at which the Laurent expansion is of the form
\begin{equation}\label{laurent exp}
f(z) = \sum_{k=-\infty}^{\infty} a_k (z -i\xi )^k,
\end{equation}
where
\begin{equation}\label{laurent coeff}
a_k =\frac{1}{2\pi i} \int_{| z-i\xi |=r} \frac{f(z)dz }{(z-i\xi
)^{k+1}},\quad k\in \Z .
\end{equation}
Then, for each $k\in \Z$,
\begin{eqnarray}\label{est}
\| \sum_{j=0}^N \frac{N!}{j!(N-j)!}  {r^{2N-2j} } a_{k-2j}\| \le 2^NM r^{N-k}  .
\end{eqnarray}
\end{lemma}
\begin{proof}
\medskip
For each $k\in \Z$ and
$0<r<1$, we have
\begin{eqnarray*}
&&\| \frac{1}{2\pi i} \int_{|z-i\xi |=r}  \left(
1+\frac{(z-i\xi )^2}{r^2}\right)^N \frac{f(z)}{(z-i\xi )^{k+1}}dz \| \\
&& \hspace{2cm} \le \frac{1}{2\pi} \int_{|z-i\xi|=r} \left|
1+\frac{(z-i\xi )^2}{r^2}\right|^N\cdot \| f(z)\| r^{-k-1}\cdot |dz | .
\end{eqnarray*}
By letting $z-i\xi = re^{i\theta}$, we have
\begin{eqnarray*}
|1+\frac{(z-i\xi )^2}{r^2}| &=&  | 1+e^{i2\theta}| \\
&=& |e^{i\theta} (e^{-i\theta}+e^{i\theta})| \\
&=& |2 \cos \theta |\\
&=& r^{-1}| 2\Re z| .
\end{eqnarray*}
Hence,
\begin{equation}\label{n15}
\left|1+\frac{(z-i\xi )^2}{r^2}\right|^N = r^{-N}2^N| \Re
z | ^N.
\end{equation}
Therefore,
\begin{eqnarray}
&&\| \frac{1}{2\pi i} \int_{|z-i\xi |=r}  \left(
1+\frac{(z-i\xi )^2}{r^2}\right)^N \frac{f(z)}{(z-i\xi )^{k+1}}dz \|   \\
&& \hspace{2cm} \le  \frac{1}{2\pi}
\int_{|z-i\xi|=r} r^{-k-1-N}  2^N| \Re
z | ^N    \frac{M}{|\Re z|^N} \cdot |dz |
\nonumber \\
&& \hspace{2cm} = \frac{ r^{-k-1-N} 2^NM}{2\pi} \int_{|z-i\xi|=r}  |dz | \nonumber \\
&& \hspace{2cm} = r^{-k-N} 2^NM.\label{c15}
\end{eqnarray}
Consider the Laurent expansion (\ref{laurent exp}).
Using (\ref{c15}) and the identity
\begin{eqnarray*}
\left(
1+\frac{(z-i\xi )^2}{r^2}\right)^N &=& 
  \sum_{j=0}^N \frac{N!}{j!(N-j)!}   \left( \frac{z-i\xi }{r}  \right)^{2j} ,
  \end{eqnarray*}
we can see that for all $k\in\Z$,
\begin{eqnarray}
\| \sum_{j=0}^N \frac{N!}{j!(N-j)!}  \frac{1}{r^{2j} } a_{k-2j}\| \le r^{-k-N} 2^NM .
\end{eqnarray}
Multiplying both sides by $r^{2N}$ gives (\ref{est}). The lemma is proven.
\end{proof}

\begin{theorem}\label{the 3.8}
Let $g\in BUC_n  (\R^+,\X)$. Then,
\begin{enumerate}
\item
If $\xi_0$ is an isolated point in $\sigma_n  (g)$, then $i\xi_0$ is either removable or a pole of ${\hat g(\lambda)}$ of order less than $n+1$;
\item If $\sigma_n  (g)=\emptyset$, then $g\in C_{0,n }(\R^+,\X)$;
\item $\sigma_n  (g)$ is a closed subset of $\R$.
\end{enumerate}
\end{theorem}
\begin{proof}
(i) By Lemma \ref{lem prep gel} we consider the Laurent expansion of $\hat g(\lambda ):=R(\lambda, \tilde{\cal D})\tilde{g}$ in a neighborhood of $i\xi_0$. Assuming that all notations are the same as in Lemma \ref{lem prep gel} with $N:=n+1$, we have  
\begin{eqnarray}
\| \sum_{j=0}^N \frac{N!}{j!(N-j)!}  {r^{2N-2j} } a_{k-2j}\| \le 2^NM r^{N-k}  .
\end{eqnarray}
Then, for all $N-k\ge 1$, that is $n\ge k$, $r^{N-k}\to 0$ as $r\to 0$. Consequently, $a_{k-2N}$, the only term without any positive power of $r$ in the left side must be zero, so
\begin{equation}
a_{k-2(n+1)}=0, \quad \mbox{whenever} \quad k\le n .
\end{equation}
This means, all coefficients $a_{-j-n-2}$ are zero for $j=0,1,2,...$. The only possible non-zero coefficients are $a_{-n-1},a_{-n},a_{-n+1},\cdots $. This means $i\xi_0$ is either removable singular point or a pole of order less than $n+1$.

\medskip
(ii) 
Before we prove the claim we recall that
$$
F(s)=e^s\frac{\Gamma (n+1,s)}{s^{n+1}}= \frac{n!}{s^{n+1}} \sum_{k=0}^n \frac{s ^k}{k!} 
$$ 
is the Laplace transform of  $f(t)=(1+t)^n$ for which
the following property is valid
\begin{equation}\label{inf be}
\lim_{s\to \pm \infty} \frac{e^s\Gamma (n+1,s)}{s^{n+1}} =0 .
\end{equation}
By a special maximum principle as proved in \cite[Lemma 4.6.6]{arebathieneu}, the entire function ${\hat g(\lambda)}$ is bounded on the strip $\{ z\in \C| \ |z|<\delta_0\} $ for any positive $\delta_0$. Moreover, by (\ref{inf be}) it is also bounded on the remaining strip $\{ z\in \C| \ |z| \ge \delta_0\}$ of the complex plane. By the Louville Theorem, ${\hat g(\lambda)}$ must be a constant. Again, by (\ref{inf be}), it must be zero. That is $g\in C_{0,n}(\R^+,\X)$.

\medskip
(iii) This is clear from the definition.
\end{proof}

\begin{corollary}\label{cor 3.9}
Let $g\in BUC_n  (\R^+,\X)$, and $i\xi_0$ be an isolated singular point of $\hat g(\lambda )$, where $\xi_0\in\R$, such that for each
\begin{equation}\label{3.22}
\lim_{\eta \downarrow 0}\eta  R(\eta +i\xi_0,\tilde{\cal D})\tilde{g} =0.
\end{equation}
Then, ${i\xi_0}$ is a removable singular point of $\hat g$.
\end{corollary}
\begin{proof}
By (i) of Theorem \ref{the 3.8}, $i\xi_0$ is a simple pole of $R(\lambda , \tilde{\cal D})\tilde{g}$, so
\begin{eqnarray*}
R(\lambda , \tilde{\cal D})\tilde{g} &=& \sum_{k=-1}^\infty a_k (\lambda -i\xi_0)^k 
\end{eqnarray*}
Therefore, for small $|\eta|$,
\begin{eqnarray*}
\eta R(\eta +i\xi_0 , \tilde{\cal D})\tilde{g} &=& \eta^j \sum_{k=-1}^\infty a_k \eta ^k \\
&=& a_{-1} + a_{0}\eta +a_{1}\eta^{2}+ \cdots 
\end{eqnarray*}
By (\ref{3.22}), $a_{-1}=0$, that means, $i\xi_0$ is a removable singular point of $\hat g$. 
\end{proof}

\section{Asymptotic behavior of solutions of fractional differential equations}
We are going to apply the spectral theory of $n $-bounded functions in the previous section to study the asymptotic behavior of mild solutions to fractional differential equations of the form
\begin{equation}\label{fde}
D^{\alpha}_Cu(t)=Au(t)+f(t), u(0)=x, 
\end{equation}
where $\alpha$ is a fixed number, $0<\alpha \le 1$, $A$ is a closed linear operator in a complex Banach space $\X$, $f$ is an element of $C_{0,n }(\R^+,\X)$. Recall that a mild solution $u$ on $\R^+$ is defined to be a continuous function $u$ on $\R^+$ such that, for each $t\in\R^+$, $J^\alpha u(t) \in D(A)$ and
\begin{equation}
u(t)=AJ^\alpha u(t)+J^\alpha f(t)+ x,
\end{equation}
for all $t\in \R^+$.
\subsection{Estimate of the spectrum of an $n $-bounded solution}
Below we will denote by
$\rho (A,\alpha)$
the set of all $\xi_0\in \C$ such that $(\xi_0^\alpha -A)$ has an inverse $(\xi_0^\alpha -A)^{-1}$ that is analytic in a neighborhood of $\xi_0$, and by $\Sigma (A,\alpha ):=\C \backslash \rho (A,\alpha).$

\medskip
We first asume that $\Re \lambda >0$. Then, for any $n $-bounded function $h$ by the proof of Lemma \ref{lem 3.5} we have
$$
\hat h (\lambda ) = (\lambda -\tilde{\cal D})^{-1}\tilde h = \tilde g,
$$
where
\begin{eqnarray*}
g(t) &=&\int^\infty_t e^{\lambda (t-s)} h(s)ds\\
&=& \int^\infty_0 e^{-\lambda \xi }h(t+\xi )d\xi .
\end{eqnarray*}
Therefore, for each $h\in BUC_n(\R^+,\X)$,
\begin{eqnarray*}
[\hat h(\lambda)](t) &=& \widetilde{ {\cal L}(S(t)h)(\lambda )}.
\end{eqnarray*}
 Next, for each $s\in \R^+$, let us denote $u_s(t):= u(t+s),f_s(t):=f(t+s)$ for all $t\ge 0$. Then, we have
\begin{eqnarray*}
u_s(t) = AJ^\alpha _s u_s(t ) +J^\alpha_s f_s(t) + u(s).
\end{eqnarray*}
Taking the Laplace transforms of both sides gives
\begin{eqnarray*}
{\cal L}u_s(\lambda) = \lambda^{-\alpha} A{\cal L} u_s(\lambda ) +\lambda^{-\alpha} {\cal L}f_s(\lambda ) + \lambda^{-1}u(s).
\end{eqnarray*}
Therefore, 
\begin{eqnarray*}
\lambda ^{1-\alpha} ( \lambda^{\alpha} -A) {\cal L} u_s(\lambda ) =\lambda^{1-\alpha} {\cal L}f_s(\lambda ) +  u(s).
\end{eqnarray*}
Next, for $\lambda$ in a neighborhood of a point $i \xi_0$ where $\xi_0 \in \rho(A,\alpha))$ and $\xi_0\not=0$,
\begin{eqnarray*}
{\cal L} u_s(\lambda ) = ( \lambda^{\alpha} -A)^{-1}     {\cal L}f_s(\lambda ) +\lambda ^{\alpha-1} ( \lambda^{\alpha} -A)^{-1} u(s).
\end{eqnarray*}
Recall that 
$$
{\cal L}(S(s)u)(\lambda ) = {\cal L}u_s(\lambda), {\cal L}(S(s)f)(\lambda ) = {\cal L}f_s(\lambda).
$$
Therefore, 
\begin{eqnarray*}
\hat u(\lambda)  = ( \lambda^{\alpha} -A)^{-1}   \hat f(\lambda )+ \lambda ^{\alpha-1} ( \lambda^{\alpha} -A)^{-1} \tilde u.
\end{eqnarray*}
As we assume that $f\in C_{0,n } (\R^+,\X)$, $\hat f(\lambda )=0$.
Hence, for $\Re\lambda >0$,
\begin{equation}\label{hatu}
\hat u(\lambda)  = \lambda ^{\alpha-1} ( \lambda^{\alpha} -A)^{-1} \tilde{u}.
\end{equation}
Below we introduce a new notation
\begin{equation}
R_\alpha (\lambda , A):= \lambda ^{\alpha-1} ( \lambda^{\alpha} -A)^{-1} .
\end{equation}
The lemma below is actually stated in \cite[Lemma 2.2]{arebat1}. For the reader's convenience we restate it in the following form with an adapted proof.
\begin{lemma}\label{lem 4.1}
Let $u\in BUC_n  (\R^+,\X)$, $\xi_0\in\R$, and let the function $G(\lambda )$ (in $\lambda$) be an analytic extension of the function $\hat u(\lambda) =(\lambda -\tilde{\cal D})^{-1}\tilde u$ with $\Re\lambda >0$ on the open disk $B( i\xi_0,r)$ with some positive $r$. Then, $G(\lambda ) =\hat u (\lambda) $ for $\Re \lambda <0$ on a disk $B(i\xi_0,r)$.
\end{lemma}
\begin{proof}
In $B(i\xi_0,r)$ the function $\lambda \mapsto (\lambda -\tilde{\cal D}) R(1,\tilde{\cal D}) G(\lambda ) $ is an analytic function. By assumption, in $B(i\xi_0)$ for $\Re\lambda >0$,
\begin{eqnarray*}
 (\lambda -\tilde{\cal D}) R(1,\tilde{\cal D}) G(\lambda ) &=&(\lambda -\tilde{\cal D}) R(1,\tilde{\cal D}) R(\lambda , \tilde{\cal D})\tilde u\\
 &=&R(1,\tilde{\cal D}) (\lambda -\tilde{\cal D}) R(\lambda , \tilde{\cal D})\tilde u\\
 &=& R(1,\tilde{\cal D}) \tilde u .
 \end{eqnarray*}
 This yields that the function $\lambda \mapsto (\lambda -\tilde{\cal D}) R(1,\tilde{\cal D}) G(\lambda ) $ is a constant function $R(1,\tilde{\cal D}) \tilde u$ on the whole $B(i\xi_0,r)$.
Therefore, if $\Re\lambda <0$,
$$R(1,\tilde{\cal D}) G(\lambda ) = R(\lambda , \tilde{\cal D}) R(1,\tilde{D})\tilde u =R(1,\tilde{\cal D})R(\lambda , \tilde{\cal D}) \tilde u.$$
 Finally, for $\Re\lambda <0$, the above identity yields 
$$
G(\lambda ) = R(\lambda , \tilde{\cal D}) \tilde u .
$$
\end{proof}

\begin{corollary}\label{cor 4.2}
Let $u\in BUC_n  (\R^+,\X)$ be a mild solution of Eq.(\ref{fde}), where $f\in C_{0,n}(\R^+,\X)$. Then,
\begin{equation}
i\sigma_n  (u) \subset \Sigma (A,\alpha )\cap i\R .
\end{equation}
\end{corollary}
\begin{proof}
By Lemma \ref{lem 4.1} it suffices to find the set of the points $i\xi$ with $\xi \in \R$ such that $\hat u(\lambda )$ with $\Re\lambda >0$ has an analytic extension to a neighborhood of $i\xi$. By (\ref{hatu}) the corollary's claim is clear.
\end{proof}
\begin{definition}
A function $h\in BUC_n(\R^+,\X)$ is said to be {\it $n$-uniformly ergodic} at $i\eta$ if 
$$
M_\eta (h):= \lim_{\alpha \downarrow 0} \alpha R(\alpha +i\eta , {\cal D})h
$$
exists as an element of $ \in BUC_n(\R^+,\X).$
\end{definition}

\begin{theorem}\label{the main}
Let $u\in BUC_n  (\R^+,\X)$ be a mild solution to Eq.(\ref{fde}) with $f\in C_{0,n}(\R^+,\X)$. Assume further that
\begin{enumerate}
\item  $\Sigma (A,\alpha )\cap i\R$ is countable;
\item $u$ is $n$-uniformly ergodic at each $i\eta$  of this set, and $M_\eta (u) =0$.
\end{enumerate}
Then, 
\begin{equation}
\lim_{t\to+\infty} \frac{1}{(1+t)^n }u(t) =0.
\end{equation}
\end{theorem}
\begin{proof}
By Corollary \ref{cor 4.2} $i\sigma_n  (u) \subset \Sigma (A,\alpha )\cap i\R$, so it is countable by our assumption. We claim further that $\sigma_n  (u)$ should be empty because of Condition (ii). In fact, as $i\sigma_n  (u)$ is at most countable and closed, if it is not empty it should have an isolated point, say $i\xi_0$. Therefore, $i\xi_0$ is an isolated singular point of $\hat u(\lambda )$. By Theorem \ref{the 3.8} $i\xi_0$ is a simple pole. However, by Condition (ii) and Corollary \ref{cor 3.9}, this simple pole is removable. That means, $i\xi_0$ is a regular point of $\hat u(\lambda )$, so $\xi_0\not\in \sigma_n  (u)$. This is a contradiction that proves that $\sigma_n  (u)$ is empty, so by Theorem \ref{the 3.8}, $u\in C_{0,n }(\R^+,\X)$. The theorem is proved.
\end{proof}

When $f=0$ we consider the well-posedness of the equation
\begin{equation}\label{eq:3}
D^{\alpha}_Cu(t)=Au(t), u(0)=x, 0<\alpha\le 1.
\end{equation}
A mild solution of (\ref{eq:3}) on an interval $\R^+$ is a continuous function $u$ on $J$ satisfying, for all $t\ge 0$, the integral equation
\begin{equation}\label{eq:2}
u(t)=x+\int_0^tg_\alpha(t-s)Au(s)ds.
\end{equation}
\begin{definition}
A family of operators $\{S_\alpha(t)\}_{t\geq 0}\subset L(\X) $ is called a resolvent operator of \ref{eq:1} if
\begin{itemize}
\item $\{S_\alpha(t)\}$ strongly continuous $t\geq 0$ and  $S_\alpha(0)=I$,
\item $S_\alpha(t)D(A)\subset D(A)$ and  $AS_\alpha(t)x=S_\alpha(t)A, x\in D(A), t\geq 0,$
\item  $S_\alpha(t)x$ is a solution of \eqref{eq:3} for all $x\in D(A)$.
\end{itemize}
\end{definition}
If Eq. (\ref{eq:3}) has a resolvent operator $S_\alpha (t)$, then (see \cite[Proposition 1.1]{pru}) every mild solution $u$ is of the form
$$
u(t)=S_\alpha (t)u(0).
$$

\begin{corollary}
Assume that Eq. (\ref{eq:3}) is well posed with a resolvent operator $(S_\alpha (t))_{t\ge 0})$ and the following conditions are satisfied
\begin{enumerate}
\item  $S_\alpha (t)$ satisfies
\begin{equation}
\sup_{t\ge 0} \frac{\| S_\alpha (t)\|}{(1+t)^n} <\infty ;
\end{equation}
\item $\Sigma (A,\alpha )\cap i\R$ is countable;
\item At each $i\zeta \in \Sigma (A,\alpha )\cap i\R$, $x\in \X$ 
\begin{equation}
\lim_{\eta\uparrow 0} \eta R_\alpha (\eta+i\zeta  , A)x =0.
\end{equation}
\end{enumerate}
Then, every mild solution $u(\cdot )=S_\alpha (\cdot )x_0\in BUC_n(\R^+,\X)$ of Eq. (\ref{eq:3}) satisfies
\begin{equation}
\lim_{t\to\infty} \frac{1}{(1+t)^n}\| u(t)\|  =0.
\end{equation}
\end{corollary}
\begin{proof}
It suffices to check the uniform ergodicity condition in Theorem \ref{the main}. By (\ref{hatu})  we have
\begin{eqnarray*}
0\le \lim_{\eta \uparrow 0} \| \eta R (\eta +i\zeta , \tilde{\cal D}) \tilde u \|_n &=& \lim_{\uparrow 0} \|  \eta \hat u (\eta +i\zeta )\|_n  \\
&=& \lim_{\eta \uparrow 0} \| \eta R_\alpha (\eta +i\zeta ,A) \widetilde {S_\alpha (\cdot )x_0}\| _n\\
&\le & \| S_\alpha (\cdot )\| _n \| \lim_{\eta \uparrow 0} \| \eta R_\alpha (\eta +i\zeta ,A) x_0\| \\
&=& 0.
\end{eqnarray*}
\end{proof}

\bibliographystyle{amsplain}

\begin{thebibliography}{10}

\bibitem{arebat1}
W. Arendt, C. J.K. Batty. Almost periodic solutions of first- and second-order Cauchy problems. {\it J. Differential Equations}, {\bf 137} (1997), no. 2, 363-383.

\bibitem{arebat}
W. Arendt, C. J.K. Batty. Tauberian theorems and stability of one-parameter semigroups. 
{\it Trans. Amer. Math. Soc.} {\bf 306} (1988), 837-852.


\bibitem{arebathieneu}
W. Arendt, C.J.K  Batty, M. Hieber, F. Neubrander, Vector-valued Laplace transforms and Cauchy problems. Second edition. Monographs in Mathematics, 96. Birkhauser/Springer Basel AG, Basel, 2011. 

\bibitem{baemeenan}
B. Baeumer, M.M. Meerschaert, and E. Nane, Brownian subordinators and fractional
Cauchy problems, {\it Trans. Am. Math. Soc.}, {\bf 361} (2009), pp. 3915-3930.

\bibitem{bas}
A. G. Baskakov,  Harmonic and spectral analysis of power bounded operators and bounded semigroups of operators on a Banach space. (Russian) {\it Mat. Zametki} {\bf 97} (2015), no. 2, 174--190; translation in {\it Math. Notes} {\bf 97} (2015), no. 1-2, 164-178


\bibitem{batengprusch}
A. Batkai, K.J. Engel, J. Pr\"uss, R. Schnaubelt, 
Polynomial stability of operator semigroups. {\it 
Math. Nachr.} {\bf 279} (2006), no. 13-14, 1425-1440.

\bibitem{batneerab1}
C. J. K. Batty, J. van Neerven, F. R\"abiger,   Tauberian theorems and stability of solutions of the Cauchy problem. {\it Trans. Amer. Math. Soc.} {\bf 350} (1998), no. 5, 2087-2103. 

\bibitem{batneerab2}
C. J. K. Batty, J. van Neerven, F. R\"abiger, Local spectra and individual stability of uniformly bounded C0-semigroups. {\it Trans. Amer. Math. Soc.} {\bf 350} (1998), no. 5, 2071-2085.

\bibitem{batyea}
C.J.K. Batty, S.B. Yeates, Stephen ,
Weighted and local stability of semigroups of operators. {\it 
Math. Proc. Cambridge Philos. Soc.} {\bf  129} (2000), no. 1, 85-98.

\bibitem{baspry}
B. Basit, A.J. Pryde, Ergodicity and stability of orbits of unbounded semigroup representations. {\it J. Aust. Math. Soc.} {\bf 77} (2004), no. 2, 209-232.

\bibitem{baz}
E. Bazhlekova, "Fractional Evolution Equations in Banach Spaces".Thesis, Technische
Universiteit Eindhoven, 2001.

\bibitem{clegrilon}
Ph. Clement, G. Gripenberg, and S.-O. Londen, Schauder estimates for equations with
fractional derivatives. {\it Trans. Am. Math. Soc.} {\bf 352} (2000), pp. 2239?2260.

\bibitem{chitom1}
R. Chill, R., Yu. Tomilov,  Stability of operator semigroups: ideas and results. In: "Perspectives in operator theory", vol. {\bf 75}, pp.71-109. Banach Center Publ. Polish Acad. Sci., Warsaw, 2007.


\bibitem{cue}
E. Cuesta, Asymptotic behaviour of the solutions of fractional integro-differential equations
and some time discretizations, {\it Discrete Contin. Dyn. Syst.} (Suppl.) (2007), pp. 277-285.

\bibitem{eidkoc}
S.D. Eidelman and A.N. Kochubei, Cauchy problem for fractional diffusion equations,
{\it J. Differential. Equations} , {\bf199}  (2004), pp. 211-255.

\bibitem{hil}
R. Hilfer, "Applications of Factional Calculus in Physics". World Scientific, River Edge, NJ,
2000.

\bibitem{hinnaiminshi}
Y. Hino, T. Naito,  Nguyen Van Minh; J.S. Shin,  "Almost periodic solutions of differential equations in Banach spaces". Stability and Control: Theory, Methods and Applications, 15. Taylor \& Francis, London, 2002.


\bibitem{keylizwar}
V. Keyantuo, C. Lizama, Carlos;M.  Warma, Existence, regularity and representation of solutions of time fractional wave equations. {\it Electron. J. Differential Equations}, 2017, Paper No. 222, 42 pp.

\bibitem{kilsritru}
A.A. Kilbas, H.M. Srivastsava, and J.J. Trujillo, "Theory and Applications of Fractional
Differential Equations", North-Holland Mathematics Studies, Vol. {\bf 204}, Elsevier Science,
Amsterdam, 2006.

\bibitem{lizngu}
C.  Lizama, G. M. N'Guerekata, Mild solutions for abstract fractional
differential equations, {\it Applicable Analysis}, {\bf 92} (2013), 1731-1754.

\bibitem{lizngu2}
C.  Lizama, G. M. N'Guerekata, Bounded mild solutions for semilinear integro
differential equations in Banach spaces, {\it Integr. Eqns Oper. Theory}, {\bf  68} (2010), pp. 207-227.

\bibitem{lyuvu}
Yu. Lyubich,   Vu Quoc Phong, 
Asymptotic stability of linear differential equations in Banach spaces. 
{\it Studia Math.}, {\bf 88} (1988), no. 1, 37-42. 

\bibitem{min} Nguyen Van Minh, A Spectral Theory of Non-Uniformly Continuous Functions and
the Loomis-Arendt-Batty-Vu Theory on the Asymptotic Behavior of Solutions of Evolution Equations. {\it J. Differential Equations}, {\bf 247} (2009), p.1249-1274.

\bibitem{minngusie}
Nguyen, Van Minh; G. N'Guerekata, S. Siegmund,  Circular spectrum and bounded solutions of periodic evolution equations. {\it J. Differential Equations} {\bf 246} (2009), no. 8, 3089-3108.

\bibitem{nee} J. van Neerven, "The Asymptotic Behaviour of
Semigroups of Linear Operators". Operator Theory: Advances and Applications, {\bf 88}. Birkhauser Verlag, Basel, 1996. 

\bibitem{pau}
L. Paunonen, Polynomial stability of semigroups generated by operator matrices. {\it J. Evol. Equ.} {\bf 14} (2014), no. 4-5, 885-911.

\bibitem{pru}
 J. Pr\"uss,
  "Evolutionary integral equations and applications". Monographs in Mathematics, {\bf 87}. Birkhauser Verlag, Basel, 1993.
  
  \bibitem{rozseista}
  J. Rozendaal, D. Seifert, David; R. Stahn,  Optimal rates of decay for operator semigroups on Hilbert spaces. {\it Adv. Math.} {\bf 346} (2019), 359-388.
  
  \bibitem{Rzesch}
  L. Rzepnicki, R. Schnaubelt,  Polynomial stability for a system of coupled strings. {\it Bull. Lond. Math. Soc.} {\bf 50} (2018), no. 6, 1117-1136.

\bibitem{sklshi}
G.M. Sklyar, V. Ya. Shirman, Asymptotic stability of a linear differential equation in a Banach space. (Russian) Teor. Funktsii Funktsional. Anal. i Prilozhen. No. 37 (1982), 127-132.

 \bibitem{yos}
  T. Yosida, "Functional Analysis". Springer-Verlag, Berlin, 1995.
  
\end{thebibliography}

\end{document}